\documentclass[12pt]{amsart}
\usepackage{amssymb,latexsym}
\usepackage{graphicx}
\usepackage{fancyhdr}
\usepackage{amsthm}
\usepackage{mathtools}
\usepackage[hyphens]{url}
\usepackage{hyperref}
\usepackage{breakurl}
\usepackage{cite}

\newcommand{\Z}{\mathbb Z}
\newcommand{\Q}{\mathbb Q}

\renewcommand{\O}{\mathcal O}

\newcommand{\op}{\operatorname}

\theoremstyle{plain}
\numberwithin{equation}{section}
\newtheorem{thm}{Theorem}[section]
\newtheorem{theorem}[thm]{Theorem}
\newtheorem{lemma}[thm]{Lemma}
\newtheorem{proposition}[thm]{Proposition}
\newtheorem{corollary}{Corollary}

\theoremstyle{definition}
\newtheorem{definition}[thm]{Definition}

\begin{document}

\setcounter{page}{1}

\title[Real cyclotomic fields of prime conductor and their class numbers]{Real cyclotomic fields of prime conductor and their class numbers}
\author{John C. Miller}
\address{Department of Mathematics\\
                Rutgers University\\
                Hill Center for the Mathematical Sciences\\
                110 Frelinghuysen Road
                Piscataway, NJ 08854-8019}
\email{jcmiller@math.rutgers.edu}


\begin{abstract}
Surprisingly, the class numbers of cyclotomic fields have only been determined for fields of small conductor, e.g. for prime conductors up to 67, due to the problem of finding the ``plus part," i.e. the class number of the maximal real subfield.  Our results have improved the situation.  We prove that the plus part of the class number is 1 for prime conductors between 71 and 151.  Also, under the assumption of the generalized Riemann hypothesis, we determine the class number for prime conductors between 167 and 241.  This technique generalizes to any totally real field of moderately large discriminant, allowing us to confront a large class of number fields not previously treatable.
\end{abstract}

\maketitle

\section{Introduction}
The cyclotomic fields are among the most intensively studied classes of number fields.  Yet their class numbers remain quite mysterious.  Surprisingly, for cyclotomic fields of prime conductor, the class number has only been determined up to conductor 67, and no further cyclotomic fields of prime conductor have had their class numbers determined unconditionally since the results of Masley \cite{Masley} more than three decades ago.  The difficulty lies in the calculation of the ``plus part" of the class number, i.e. the class number of the maximal real subfield, a problem that has been described as ``notoriously hard" \cite{Schoof}.  For fields of larger conductor, their Minkowski bounds are far too large to be useful, and their discriminants are too large for their class numbers to be treated by Odlyzko's discriminant bounds.

Our results have improved the situation.  Following the method introduced in the author's earlier paper \cite{Miller}, we overcome the problem of large discriminants by establishing nontrivial lower bounds for sums over the prime ideals of the Hilbert class field, allowing us to obtain an upper bound for the class number.  We have the following main result.

\begin{theorem}\label{MainResult}
Let $p$ be a prime integer, and  let $\Q(\zeta_p + \zeta_p^{-1})$ denote the maximal real subfield of the $p$-th cyclotomic field $\Q(\zeta_p)$.  Then the class number of $\Q(\zeta_p + \zeta_p^{-1})$ is $1$ for $p \leq 151$.

Furthermore, under the assumption of the generalized Riemann hypothesis, the class number $h$ of $\Q(\zeta_p+\zeta_p^{-1})$ is
$$h = \begin{cases}
1 & \text{if } p \leq 241 \text{ and } p \neq 163, 191, 229, \\
4 & \text{if } p = 163, \\
11 & \text{if } p = 191, \\
3 & \text{if } p = 229. \\
  \end{cases}$$
\end{theorem}

\section{Upper bounds for class numbers using Odlyzko's discriminant bounds}
Odlyzko's discriminant lower bounds can be applied to find upper bounds for class numbers.  Further details can be found in \cite{Linden}, \cite{Masley} and \cite{Odlyzko}.

\begin{definition}
Let $K$ denote a number field of degree $n$ over $\mathbb Q$.  Let $d(K)$ denote its discriminant.  The \emph{root discriminant} $\operatorname{rd}(K)$ of $K$ is defined to be:
$$\operatorname{rd}(K)=|d(K)|^{1/n}.$$
\end{definition}

\begin{proposition}
Let $L/K$ be an extension of number fields.  Then
$$\operatorname{rd}(K) \leq \operatorname{rd}(L),$$ with equality if and only if $L/K$ is unramified at all finite primes.
\end{proposition}

\begin{proof}
The discriminants of $K$ and $L$ are related by the formula
$$|d(L)| = N(d(L/K))|d(K)|^{[L:K]}$$
where $d(L/K)$ denotes the relative discriminant ideal and $N$ denotes the absolute norm of the ideal, from which the first statement follows.  A prime of $K$ ramifies in $L$ if and only if the prime divides the relative discriminant $d(L/K)$.  Thus, $L/K$ is unramified at all finite primes if and only if $d(L/K)$ is the unit ideal, proving the second statement.
\end{proof}

\begin{corollary}
Let $K$ be a number field.  Then the Hilbert class field of $K$ has the same root discriminant as $K$.
\end{corollary}

Now suppose $K$ is a totally real number field of degree $n$.  Odlyzko constructed a table \cite{OdlyzkoTable} of pairs $(A,E)$ such that the discriminant of $K$ has the lower bound
$$|d(K)| > A^n e^{-E}.$$
Masley \cite{Masley} and van der Linden \cite{Linden} applied Odlyzko's discriminant bounds to the Hilbert class field and used the corollary above to get
$$ \log{\op{rd}(K)} > \log{A} - \frac{E}{hn}.$$
If $\op{rd}(K) < A$, then we obtain an upper bound for the class number $h$,
$$h < \frac{E}{n(\log{A} - \log{\op{rd}(K))}}.$$
However, if the root discriminant of $K$ is larger than the largest $A$ in Odlyzko's table, the above method can not be applied.  The largest value for $A$ in Odlyzko's table is $60.704$ (or $213.626$ if the generalized Riemann hypothesis is assumed).

\section{Upper bounds for class numbers beyond Odlyzko's discriminant bounds}
We may obtain an upper bounds for class numbers of number fields of root discriminant larger than $60.704$ by establishing lower bounds for sums over the prime ideals of the Hilbert class field.  The author's earlier paper \cite{Miller} treats this in detail.  We repeat here two lemmas that will be useful.

If we do not assume the generalized Riemann hypothesis, we have the following lemma.
\begin{lemma}\label{SplitPrimeLemmaNoRH}
Let $K$ be a totally real field of degree $n$, and let
$$F(x) = \frac{e^{-\left(\frac{x}{c}\right)^2}}{\cosh{\frac{x}{2}}}$$
for some positive constant $c$.  Suppose $S$ is a subset of the prime integers which totally split into principal prime ideals of $K$.  Let
\begin{align*}
B &= \frac{\pi}{2} + \gamma + \log{8\pi}  - \log \operatorname{rd}(K) - \int_0^\infty \frac{1-F(x)}{2}\left( \frac{1}{\sinh{\frac{x}{2}}} +\frac{1}{\cosh{\frac{x}{2}}} \right) \, dx \\ & \quad + 2 \sum_{p \in S} \sum_{m=1}^\infty \frac{\log p}{p^{m/2}}F(m \log p).
\end{align*}
If $B > 0$ then we have an upper bound for the class number $h$ of $K$,
$$h < \frac{2c \sqrt{\pi}}{nB}.$$
\end{lemma}

On the other hand, if we assume the truth of the generalized Riemann hypothesis, we have the following lemma.
\begin{lemma}\label{SplitPrimeLemma}
Let $K$ be a totally real field of degree $n$, and let
$$F(x) = e^{-(\frac{x}{c})^2}$$
for some positive constant $c$.  Suppose $S$ is a subset of the prime integers which totally split into principal prime ideals of $K$.  Let
\begin{align*}
B &= \frac{\pi}{2} + \gamma + \log{8\pi}  - \log \operatorname{rd}(K) - \int_0^\infty \frac{1-F(x)}{2}\left( \frac{1}{\sinh{\frac{x}{2}}} + \frac{1}{\cosh{\frac{x}{2}}} \right) \, dx \\ & \quad + 2 \sum_{p \in S} \sum_{m=1}^\infty \frac{\log p}{p^{m/2}}F(m \log p).
\end{align*}
If $B > 0$ then we have, under the generalized Riemann hypothesis, an upper bound for the class number $h$ of $K$,
$$h < \frac{2c \sqrt{\pi} e^{(c/4)^2}}{nB}.$$
\end{lemma}

Given an element $x$ of a Galois number field $K$, we define its \emph{norm} to be
$$N(x) = \left | \prod_{\sigma \in \op{Gal}(K/\Q)} \sigma(x) \right |.$$
Note that if $x$ is in the ring of integers of $K$, and if its norm is a prime integer $p$ which is unramified in $K$, then $p$ totally splits into principal ideals, and we can take $p$ to be in the set $S$ above.  Once we find sufficiently many such prime integers which totally split into principal ideals, we can establish an upper bound for the class number.

Once upper bound is established, we can frequently use various ``push up" and ``push down" lemmas \cite{Linden, Masley} to pin down the exact class number.  However, our preference will be to appeal to the results of Schoof \cite{Schoof}.  In his ``Main Table," for each prime conductor less than $10,000$, he gives a number $\tilde{h}$ such that the class number $h$ either equals $\tilde{h}$, or $h > 80000 \cdot \tilde{h}$.  In particular, if our upper bound for $h$ is less than $80000$, then we have $h = \tilde{h}$.

\section{The class number of $\Q(\zeta_p+ \zeta_p^{-1})$ for primes $p = 71, 73, 79, 83$}
We recall a few facts about real cyclotomic fields of prime conductor.  Let $p$ be a prime integer, and let $\Q(\zeta_p+\zeta_p^{-1})$ denote the $p$th real cyclotomic field, i.e. the maximal real subfield of the cyclotomic field $\Q(\zeta_p)$, where $\zeta_p$ is a primitive $p$th root of unity.  The degree $n$ of  $\Q(\zeta_p+\zeta_p^{-1})$ is $n=(p-1)/2$, and its discriminant is given by
$$d(\Q(\zeta_p+\zeta_p^{-1})) = p^{\frac{1}{2}(p-3)}.$$
Thus, its root discriminant is
$$\op{rd}(\Q(\zeta_p+\zeta_p^{-1})) = p^{\frac{p-3}{p-1}}.$$
The prime integers which totally split in this field are precisely those which are congruent to $\pm 1$ modulo $p$.

The ring of integers of $\Q(\zeta_p+ \zeta_p^{-1})$ is $\Z[\zeta_p+ \zeta_p^{-1}]$.  Until otherwise noted, the integral basis that we will use is $\{b_0, b_1,\dots, b_{n-1}\}$, with $b_0 = 1$ and $b_j = 2 \cos{\frac{2 \pi j}{p}}$ for $j=1,\dots,n-1$.

For $p=71$, the root discriminant is approximately $62.85$, which too large for the Odlyzko bounds to be useful.  Our goal is to find to sufficiently many algebraic integers $x$ in $\Q(\zeta_{71} + \zeta_{71}^{-1})$ such that $N(x)$ is a prime integer congruent to $\pm 1$ modulo $p$, and then apply Lemma \ref{SplitPrimeLemmaNoRH}.

\begin{lemma}\label{PrimeLemma71}
In the real cyclotomic field $\Q(\zeta_{71}+\zeta_{71}^{-1})$, there exist algebraic integers of norms $283$, $569$, $709$, $853$, $1277$, $1279$, $1847$, $1987$, $2129$, and $2131$, i.e. the ten smallest prime integers that are congruent to $\pm 1$ modulo $71$.
\end{lemma}

\begin{proof}
Let $b_0 = 1$ and $b_j = 2 \cos{\frac{2 \pi j}{71}}$ for $j$ from $1$ to $34$.  Then $\{b_0, b_1,\dots, b_{34}\}$ is an integral basis of $\Q(\zeta_{71}+\zeta_{71}^{-1})$.  We search over ``sparse vectors," where almost all the coefficients are zero, and the remaining coefficients are $\pm 1$.  We find the following ten elements and their norms:

\begin{center}
\begin{tabular}{| l | r| }
\hline
  Element & Norm \\ \hline
  $b_0 + b_1 - b_6$                  & 283 \\
  $b_0 + b_1 + b_8$                 & 569 \\
  $b_0 + b_1 + b_4 - b_{22}$  & 709 \\
  $b_0 + b_1 - b_5$                  & 853 \\
  $b_0 + b_1 + b_3 + b_{12}$ & 1277 \\
  $b_0 + b_1 - b_8 - b_{13}$   & 1279 \\
  $b_0 + b_1 + b_2 - b_{28}$  & 1847 \\
  $b_0 + b_1 - b_{13}$             & 1987\\
  $b_0 + b_1 - b_3 + b_{10}$  & 2129\\
  $b_0 + b_1 - b_{27}$             & 2131\\
\hline
\end{tabular}
\end{center}

\end{proof}

\begin{proposition}\label{Prop71}
The class number of $\Q(\zeta_{71}+\zeta_{71}^{-1})$ is $1$.
\end{proposition}

\begin{proof}
Let $F$ be the function
$$F(x) = \frac{e^{-\left(\frac{x}{c}\right)^2}}{\cosh{\frac{x}{2}}}$$
with $c = 15$.  We have the following lower found for the contribution of prime ideals,
$$ 2 \sum_{p \in S} \sum_{m=1}^\infty \frac{\log p}{p^{m/2}}F(m \log p) > 2 \sum_{p \in S} \frac{\log p}{\sqrt{p}}F(\log p) > 0.2448.$$
The following integral can be calculated using numerical integration:
$$\int_0^\infty \frac{1-F(x)}{2}\left( \frac{1}{\sinh{\frac{x}{2}}} + \frac{1}{\cosh{\frac{x}{2}}} \right) \, dx <1.2964.$$
We have
\begin{align*}
B = \frac{\pi}{2} &+ \gamma + \log{8\pi}  - \log \operatorname{rd}(K) - \int_0^\infty \frac{1-F(x)}{2}\left( \frac{1}{\sinh{\frac{x}{2}}} + \frac{1}{\cosh{\frac{x}{2}}} \right) \, dx \\& + 2 \sum_{p \in S} \sum_{m=1}^\infty \frac{\log p}{p^{m/2}}F(m \log p) > 0.1796.
\end{align*}
By Lemma \ref{SplitPrimeLemmaNoRH}, the class number is less than $9$.  Applying Schoof's table \cite{Schoof}, or using more elementary algebraic arguments, we find that the class number is $1$.
\end{proof}

The proofs for conductors $73$, $79$ and $83$ are entirely similar.  It is straightforward to find algebraic integers with the 10 smallest prime norms congruent to $\pm 1$ modulo $p$.  This is sufficient to get an upper bound less than $80000$.  Using Schoof's table, we find that each has class number $1$.

\section{The class number of $\Q(\zeta_p + \zeta_p^{-1})$, prime $p$, $89 \leq p \leq 131$}
We first consider the conductor $131$.  The root discriminant of $\Q(\zeta_{131} + \zeta_{131}^{-1})$ is approximately $121.53$.  Quite a few prime ideals will be required in order to have a sufficiently large contribution to establish an upper bound for the class number.

\begin{proposition}\label{Prop131}
The class number of $\Q(\zeta_{131}+\zeta_{131}^{-1})$ is $1$.
\end{proposition}

\begin{proof}
Let $n$ denote the degree of the field.  We calculate the norm of every element of the form
$$x = b_0 + b_1 + a_1 b_{j_1} + a_2 b_{j_2} + a_3 b_{j_3} + a_4 b_{j_4} + a_5 b_{j_5} + a_6 b_{j_6},$$
where $1 < j_1 < j_2 < j_3 < j_4 < j_5 < j_6 < n$ and $a_j \in \{-1, 0, 1\}$ for $1 \leq j \leq 6$.  By searching these sparse vectors $x$ and calculating their norms, we find 12,087 prime integers that are less than 20,000,000 and congruent to $\pm 1$ modulo $131$.

Let $S$ be the set of those 12,087 primes, and let
$$F(x) = \frac{e^{-\left(\frac{x}{c}\right)^2}}{\cosh{\frac{x}{2}}}$$
with $c = 1000$.  We have the following lower found for the contribution of prime ideals,
$$ 2 \sum_{p \in S} \sum_{m=1}^\infty \frac{\log p}{p^{m/2}}F(m \log p) > 2 \sum_{p \in S} \sum_{m=1}^2 \frac{\log p}{p^{m/2}}F(m \log p) > 0.708.$$
By numerical integration, we have
$$\int_0^\infty \frac{1-F(x)}{2}\left( \frac{1}{\sinh{\frac{x}{2}}} + \frac{1}{\cosh{\frac{x}{2}}} \right) \, dx < 1.264.$$
We have $B > 0.015$, so by applying  Lemma \ref{SplitPrimeLemmaNoRH} we find that the class number is less than $3636$.  We can now use Schoof's table to find that the class number is $1$.
\end{proof}

The proof for prime conductors between $89$ and $127$ is entirely similar, and we find that all have class number $1$.

\section{The class number of $\Q(\zeta_p+ \zeta_p^{-1})$ for primes $p = 137, 139, 149, 151$}
As the root discriminant of the fields increases, so does the required contribution from the prime ideals.  To find sufficiently many split primes in fields of larger discriminant, it is often quicker to additionally search over sparse vectors using an alternative basis.  A useful alternative to the basis $b_0, b_1, \dots, b_{n-1}$ is
$$c_k = \sum_{j=0}^k b_j, \quad k=0, 1,\dots, n-1.$$

\begin{proposition}\label{Prop151}
The class number of $\Q(\zeta_{151}+\zeta_{151}^{-1})$ is $1$.
\end{proposition}

\begin{proof}
Let $n$ denote the degree of the field, and let $\O$ denote the ring of integers.

We first consider all $x \in \O$ of the forms
$$x = b_0 + b_1 + a_1 b_{j_1} + a_2 b_{j_2} + a_3 b_{j_3} + a_4 b_{j_4} + a_5 b_{j_5} + a_6 b_{j_6},$$
and
$$x = b_1 + a_1 b_{j_1} + a_2 b_{j_2} + a_3 b_{j_3} + a_4 b_{j_4} + a_5 b_{j_5} + a_6 b_{j_6},$$
where $1 < j_1 < j_2 < j_3 < j_4 < j_5 < j_6 < n$ and $a_j \in \{-1, 0, 1\}$ for $1 \leq j \leq 6$.

We also search over the alternative basis,
$$x = c_0 + a_1 c_{k_1} + a_2 c_{k_2} + a_3 c_{k_3} + a_4 c_{k_4} + c_5 b_{k_5} + c_6 b_{k_6},$$
where $1 \leq k_1 < k_2 < k_3 < k_4 < k_5 < k_6 < n$ and $a_k \in \{-1, 0, 1\}$ for $1 \leq k \leq 6$.

Let $T$ denote the set of all such elements $x$.

The ideal $(p)$ is totally ramified.  Thus, if $x \in \O$ has norm $N(x)$ divisible by $p$, we can divide $x$ by any element of norm $p$, say $2b_0 - b_1$, to get an algebraic integer $(2b_0 - b_1)^{-1} x$ with norm $N(x)/p$.  Therefore consider the non-$p$ parts of all norms $N(x)$.  Initially we search these sparse vectors to find norms which are less than $10^{15}$ and congruent to $\pm 1$ modulo $151$.  We define the set $U$ to be
$$U = \{\text{non-}p\text{ part of } N(x) | x \in T, N(x) < 10^{15} \}.$$

Let $S_1$ be the set of primes
$$S_1 = \{m : m \in U, m \text{ prime, }  m \equiv \pm 1 \,(\bmod \,151)\}.$$

Unfortunately, we find that primes of $S_1$ make an insufficient contribution.  We could attempt to search over sparse vectors with more nonzero coefficients, but this is very time consuming.  Instead we find elements of larger norm, and take quotients, as will be described below.

 Let $S_2$ be the set of primes defined by
$$S_2 = \{p : pq \in U, p \text{ prime, } p \notin S_1, q \in S_1  \},$$
noting that if $N(x) = pq$ and $N(y)=q$, for $x, y \in \O$, then $\frac{x}{\sigma(y)}$ is in $\O$ with norm $p$ for some Galois automorphism $\sigma$.

Now put $S = S_1 \cup S_2$ and $c = 115$.   We have the following lower found for the contribution of prime ideals,
$$ 2 \sum_{p \in S} \sum_{m=1}^\infty \frac{\log p}{p^{m/2}}F(m \log p) > 2 \sum_{p \in S} \sum_{m=1}^2 \frac{\log p}{p^{m/2}}F(m \log p) > 0.8745.$$
Applying Lemma \ref{SplitPrimeLemmaNoRH}, we have $B > 0.0316$, so the class number is less than $171$.  We can now use Schoof's table to find the class number is $1$.
\end{proof}

The proofs for conductors $137$, $139$ and $149$ are entirely similar.  We determine that all have class number $1$, and have proved the first statement of Theorem \ref{MainResult}.

\section{The class number of $\Q(\zeta_p + \zeta_p^{-1})$ for primes $p$, $167 \leq p \leq 193$}
For prime conductors greater than $151$, we will assume the generalized Riemann hypothesis.  Using Odlyzko's discriminant bounds, van der Linden proved (under GRH) that the class number of $\Q(\zeta_p + \zeta_p^{-1})$ is $1$ for $p < 163$ and is $4$ for $p = 163$ \cite{Linden}.

Although the question of the class number of $\Q(\zeta_{167} + \zeta_{167}^{-1})$ has remained open, we can find its class number as a direct consequence of Schoof's results \cite{Schoof}.  Indeed, the root discriminant of $\Q(\zeta_{167} + \zeta_{167}^{-1})$ is only $162.93...$, so it is small enough for Odlyzko's discriminant bounds to establish an upper bound for the class number, without any recourse to knowledge of the prime ideals.  Choosing the pair $(A,E) = (170.633,  4790.3)$ from Odlyzko's table ``Table 3: GRH Bounds for Discriminants" \cite{OdlyzkoTable}, we get an upper bound for the class number $h$,
$$h < \frac{E}{n(\log{A} - \log{\op{rd}(\Q(\zeta_{167} + \zeta_{167}^{-1}))})} < 1208.$$
Since $h < 80000$, we can use Schoof's table to prove $h = 1$ for $\Q(\zeta_{167} + \zeta_{167}^{-1})$.

Entirely similarly, we can use the Odlyzko GRH bounds together with Schoof's table to determine the class numbers of prime conductor between $173$ and $193$, all of which are $1$ except for $\Q(\zeta_{191} + \zeta_{191}^{-1})$ which has class number $11$.

\section{The class number of $\Q(\zeta_p + \zeta_p^{-1})$ for primes $p$, $197 \leq p \leq 241$}
The root discriminant of $\Q(\zeta_{197} + \zeta_{197}^{-1})$ is only $186.66...$, so we can also apply Odlyzko's GRH bounds here to get an upper bound of $h < 152927$.  Unfortunately, this bound is not less than $80000$, so we are not yet able use Schoof's table.  To get a better upper bound for $h$, we will study the prime ideals of the field and apply Lemma \ref{SplitPrimeLemma}.

\begin{proposition}
Under the generalized Riemann hypothesis, the class number of \\$\Q(\zeta_{197} + \zeta_{197}^{-1})$ is $1$.
\end{proposition}

\begin{proof}
Searching for an algebraic integer with small prime norm congruent to $\pm 1$ modulo $197$, we find the element
$$b_0 + b_1 + b_2 - b_{10} + b_{33} - b_{70} - b_{83}$$
which has norm $1181$.

Assuming the generalized Riemann hypothesis, we apply Lemma \ref{SplitPrimeLemma} using $S=\{1181\}$ and $c = 9.25$ to prove that the class number is less than $1027$.  Now we use Schoof's table to find that the class number is $1$. 
\end{proof}

\begin{proposition}
Under the generalized Riemann hypothesis, the class number of \\$\Q(\zeta_{199} + \zeta_{199}^{-1})$ is $1$.
\end{proposition}

\begin{proof}
As in the proof of the previous proposition, we find two algebraic integers of small prime norm congruent to $\pm 1$ modulo $199$:
$$N(b_0 + b_1 - b_{14} - b_{23} + b_{59} - b_{77} ) = 29453,$$
 and
$$N(b_0 + b_1 + b_6 - b_{35} + b_{54} - b_{56} - b_{96}) = 26267.$$
We apply Lemma \ref{SplitPrimeLemma} using $S=\{26267, 29453\}$ and $c = 11.5$ to show that the class number is less than $47719$, and use Schoof's table to prove that the class number is $1$. 
\end{proof}

\begin{proposition}
Under the generalized Riemann hypothesis, the class number of \\$\Q(\zeta_{211} + \zeta_{211}^{-1})$ is $1$.
\end{proposition}

\begin{proof}
We find an algebraic integer of small prime norm congruent to $\pm 1$ modulo $211$:
$$N(b_0 + b_1 - b_8 - b_{60} + b_{64} + b_{67} ) = 2111.$$
We apply Lemma \ref{SplitPrimeLemma} using $S=\{2111\}$ and $c = 10.75$ to get that the class number is less than $13476$, and use Schoof's table to find the class number is $1$. 
\end{proof}

For fields of larger discriminant, it is more difficult to find split primes of small norm.  Often the quickest approach is to additionally search over sparse vectors using a different basis, and then take quotients of appropriately chosen algebraic integers.  We use the alternative basis
$$c_k = \sum_{j=0}^k b_j, \quad k=0, 1,\dots, n-1$$
given earlier.  This approach is used in the proof of the following proposition.

\begin{proposition}
Under the generalized Riemann hypothesis, the class number of \\$\Q(\zeta_{223} + \zeta_{223}^{-1})$ is $1$.
\end{proposition}

\begin{proof}
Searching over the alternative basis $(c_k)$, we find the algebraic integer
$$\alpha = c_0-c_6-c_{26}-c_{77}+c_{99}$$
which has norm $6689 \cdot 42284369$.
Searching over the usual basis $(b_j)$, we find
$$\beta = b_1 + b_{11} + b_{30} + b_{95}$$
which has norm $42284369$.
For some Galois automorphism $\sigma$, the quotient $\gamma = \frac{\sigma(\alpha)}{\beta}$ is an algebraic integer of norm $6689$.
We also can find the algebraic integer
$$\delta = c_0 + c_6 - c_{11} - c_{15} + c_{25}$$
which has norm $2677 \cdot 6689$.  Taking quotients again using a (possibly different) Galois automorphism $\sigma$, we can find an algebraic integer $\frac{\sigma(\delta)}{\gamma}$ of norm  $2677$.
Letting $S = \{2677, 6689\}$ and $c=10.5$, we can apply Lemma \ref{SplitPrimeLemma} to find that the class number is less than $6762$, and we use Schoof's table \cite{Schoof} to find the class number is $1$. 
\end{proof}

\begin{proposition}
Under the generalized Riemann hypothesis, the class number of \\$\Q(\zeta_{227} + \zeta_{227}^{-1})$ is $1$.
\end{proposition}

\begin{proof}
Searching over sparse vectors, using our two bases $(b_j)$ and $(c_k)$, we find the following elements and their norms:
\begin{center}
\begin{tabular}{| l | c| }
\hline
  Element & Norm \\ \hline
  $b_0 + b_1 - b_{21} + b_{75} - b_{96} - b_{112}$	& $4053311$\\
  $c_0 + c_4 + c_{35} - c_{56} + c_{83}$			& $4053311 \cdot 7717$\\
  $c_0 + c_{40} + c_{68} + c_{77} + c_{83}$		& $7717 \cdot 20431$\\
  $b_1 - b_{12} - b_{41} - b_{53}$				& $20431 \cdot 1361$\\
\hline
\end{tabular}
\end{center}
By successively taking quotients by the appropriate Galois conjugates, we can find algebraic integers of norms $7717$, $20431$ and $1361$.

Setting $S = \{1361, 7717, 20431\}$ and $c=9.75$, we apply Lemma \ref{SplitPrimeLemma} to get a class number upper bound of $1431$.  Using Schoof's table, we find the class number is $1$.
\end{proof}

\begin{proposition}
Under the generalized Riemann hypothesis, the class number of \\$\Q(\zeta_{229} + \zeta_{229}^{-1})$ is $3$.
\end{proposition}

\begin{proof}
Searching over sparse vectors, using our two bases $(b_j)$ and $(c_k)$, we find the following elements and their norms:
\begin{center}
\begin{tabular}{| l | c| }
\hline
  Element & Norm \\ \hline
  $c_0 - c_4 - c_{10} - c_{41} + c_{106} + c_{112}$	& $4887317$\\
  $b_1 - b_6 + b_{54} + b_{107}$				& $4887317 \cdot 2699453$\\
  $c_0 + c_{10} + c_{53} + c_{71} - c_{79} + c_{87}$	& $2699453 \cdot 49463$\\
  $c_0 - c_{13} + c_{14} - c_{63} - c_{77} + c_{79}$	& $49463 \cdot 207017$\\
  $b_0 + b_1 - b_2 - b_{16} - b_{72} + b_{88}$		& $207017 \cdot 43051$\\
  $c_0 - c_3 - c_{41} - c_{45} + c_{67}$			& $43051 \cdot 6871$\\
  $c_0 - c_4 - c_{41} + c_{53} + c_{96}$			& $6871 \cdot 2749$\\
\hline
\end{tabular}
\end{center}
By successively taking quotients by the appropriate Galois conjugates, we can find algebraic integers of norms $6871$ and $2749$.

Setting $S = \{2749, 6871\}$ and $c=11$, we apply Lemma \ref{SplitPrimeLemma} to show a class number upper bound of $12734$.  Using Schoof's table, we prove the class number is $3$.
\end{proof}

\begin{proposition}
Under the generalized Riemann hypothesis, the class number of \\$\Q(\zeta_{233} + \zeta_{233}^{-1})$ is $1$.
\end{proposition}

\begin{proof}
Searching over sparse vectors, using our two bases $(b_j)$ and $(c_k)$, we find the following elements and their norms:
\begin{center}
\begin{tabular}{| l | c| }
\hline
  Element & Norm \\ \hline
  $b_0 + b_1 + b_{14} + b_{69}$				& $53591$\\
  $b_1 + b_6 - b_{21} + b_{77} - b_{114}$			& $53591 \cdot 76423$\\
  $b_0 + b_1 - b_8 - b_{40} + b_{47} - b_{86}$		& $76423 \cdot 174749$\\
  $b_0 + b_1 + b_{20} + b_{95}$				& $174749 \cdot 467$\\
\hline
\end{tabular}
\end{center}
By successively taking quotients by the appropriate Galois conjugates, we can find algebraic integers of norm $467$.

Setting $S = \{467\}$ and $c=9.5$, we apply Lemma \ref{SplitPrimeLemma} to find a class number upper bound of $1450$.  Using Schoof's table, we prove the class number is $1$.
\end{proof}

\begin{proposition}
Under the generalized Riemann hypothesis, the class number of \\$\Q(\zeta_{239} + \zeta_{239}^{-1})$ is $1$.
\end{proposition}

\begin{proof}
Searching over sparse vectors, using our two bases $(b_j)$ and $(c_k)$, we find the following elements and their norms:
\begin{center}
\begin{tabular}{| l | c| }
\hline
  Element & Norm \\ \hline
  $b_0 + b_1 - b_{10} - b_{18} - b_{106} - b_{107} + b_{116}$	& $4423479397$\\
  $b_1 - b_2 - b_{33} - b_{60} + b_{83} - b_{84}$			& $4423479397 \cdot 2389$\\
  $b_0 + b_1 + b_{22} - b_{46}$						& $2389 \cdot 271981$\\
  $c_0 - c_{11} + c_{21} - c_{93} + c_{104}$				& $271981 \cdot 10993$\\
\hline
\end{tabular}
\end{center}
By successively taking quotients by the appropriate Galois conjugates, we can find algebraic integers of norms $2389$, $271981$ and $10993$.

Setting $S = \{2389, 10993, 271981\}$ and $c=11.5$, we apply Lemma \ref{SplitPrimeLemma} to establish a class number upper bound of $32486$.  Using Schoof's table, we find the class number is $1$.
\end{proof}

\begin{proposition}
Under the generalized Riemann hypothesis, the class number of \\$\Q(\zeta_{241} + \zeta_{241}^{-1})$ is $1$.
\end{proposition}

\begin{proof}
Searching over sparse vectors, using our two bases $(b_j)$ and $(c_k)$, we find the following elements and their norms:
\begin{center}
\begin{tabular}{| l | c| }
\hline
  Element & Non-241 part of norm \\ \hline
  $b_0 + b_1 - b_{37} - b_{52} - b_{118}$					& $5926189$\\
  $c_0 + c_8 - c_{12} + c_{43} - c_{47} + c_{99}$			& $5926189 \cdot 87487819$\\
  $b_0 + b_1 + b_{10} + b_{11} + b_{49} + b_{56} - b_{117}$	& $87487819 \cdot 47237$\\
  $c_0 - c_6 - c_{63} - c_{67} + c_{68} $					& $47237 \cdot 12049$\\
  $b_1 - b_2 + b_3 - b_{17} - b_{44} + b_{61}$				& $12049 \cdot 68927$\\
  $c_0 - c_{15} + c_{46} - c_{65} + c_{66} $				& $68927 \cdot 56393$\\
  $c_0 + c_3 + c_{42} + c_{53} + c_{95} + c_{100}$			& $56393 \cdot 5783$\\
  $b_1 + b_2 + b_{61} + b_{76} - b_{104}$				& $5783 \cdot 1447$\\
\hline
\end{tabular}
\end{center}
Note that although the norm of $b_1 - b_2 + b_3 - b_{17} - b_{44} + b_{61}$ is actually $12049 \cdot 68927 \cdot 241$, we can divide by any element of norm 241, such as $2b_0 - b_1$, to get an algebraic integer of norm $12049 \cdot 68927$.
 
By successively taking quotients by the appropriate Galois conjugates, we can find algebraic integers of norms $12049$, $5783$ and $1447$.

Setting $S = \{1447, 5783, 12049\}$ and $c=10$, we apply Lemma \ref{SplitPrimeLemma} to show a class number upper bound of $2153$.  Using Schoof's table \cite{Schoof}, we prove the class number is $1$.
\end{proof}

This completes the proof of Theorem \ref{MainResult}.

\section{Concluding remarks}
It is remarkable to observe the power of analytic objects -- the L-functions and their explicit formulae -- when brought to bear on the essentially algebraic problem of determining the class number of a real cyclotomic field.  For example, to unconditionally prove that $\Z[\zeta_{83} + \zeta_{83}^{-1}]$ is a principal ideal domain using the explicit formula, we really only need to prove three prime ideals are principal, which stands in stark contrast to using the Minkowski bound, which require approximately $10^{20}$ primes to be checked.

The techniques used in this paper may be used to calculate upper bounds of class numbers of other totally real number fields of moderately large discriminant, allowing us to tackle the class number problem for a large group of number fields which previously had not been treatable by any known methods.  For example, the calculation of the class numbers of the real cyclotomic fields of conductors 256 and 512 can be found in the author's earlier paper \cite{Miller}.

\section{Acknowledgments}
I would like to thank my advisor, Henryk Iwaniec, for introducing me to the class number problems of cyclotomic fields, and for his steadfast encouragement.  I appreciate the interest and feedback of Stephen D. Miller, Carl Pomerance, Ren\'e Schoof and Christopher Skinner.  Finally, I am exceedingly thankful for the careful reading and suggestions of Jerrold Tunnell, Lawrence Washington, and the anonymous referee.

\medskip

\end{document}